\documentclass[a4paper,12pt]{article}
\usepackage[T1]{fontenc}
\usepackage[utf8]{inputenc}
\usepackage{graphicx}
\usepackage{fancyhdr}
\usepackage{float}
\usepackage{xcolor}
\usepackage{authblk}
\usepackage{mathrsfs}
\usepackage{empheq}
\usepackage[hyphens]{url}
\usepackage{hyperref} 
\usepackage[]{breakurl}
\usepackage[]{amsthm}
\usepackage{tikz}
\usepackage{bm}
\usepackage{amsmath}
\usepackage{amssymb}
\usepackage{url}
\usepackage{cite}

\newtheorem{theorem}{Theorem}[section]
\newtheorem{corollary}{Corollary}[theorem]
\newtheorem{lemma}[theorem]{Lemma}

\usepackage{geometry}
\begin{document}

\title{Integral representations of the Riemann zeta function of odd argument}

\author[$\dagger$]{Jean-Christophe {\sc Pain}$^{1,2,}$\footnote{jean-christophe.pain@cea.fr}\\
\small
$^1$CEA, DAM, DIF, F-91297 Arpajon, France\\
$^2$Universit\'e Paris-Saclay, CEA, Laboratoire Mati\`ere en Conditions Extr\^emes,\\ 
F-91680 Bruy\`eres-le-Ch\^atel, France
}

\maketitle

\begin{abstract}
In this article we obtain, using an expression of the digamma function $\psi(x)$ due to Mikolas, integral representations of the zeta function of odd arguments $\zeta(2p+1)$ for any positive value of $p$. The integrand consists of the product of a polynomial by one or two elementary trigonometric functions. Examples for the first values of the argument are given. Some of them were already derived by other methods.
\end{abstract}

\section{Introduction}

The Riemann zeta function \cite{Edwards1974,Cartier1992} at odd integers is still an active field of investigation \cite{Ghusayni2009,Rivoal2020} and rather little is known about it. Let us mention a short (non-exhaustive) list of important results. In 1979, Ap\'ery proved that $\zeta(3)$ is irrational \cite{Apery1979,Poorten1979} and in 2000 Rivoal found that infinitely many of the numbers $\zeta(2n + 1)$, $n\geq 1$, are irrational \cite{Rivoal2000} and even linearly independent over $\mathbb{Q}$. Zudilin proved in Ref. \cite{Zudilin2001} the existence of at least one irrational number amongst $\zeta(5)$, $\zeta(7)$, $\zeta(9)$, $\zeta(11)$, and in 2002 Rivoal showed that at least one of the nine numbers $\zeta(5)$, $\zeta(7), \cdots, \zeta(21)$ was irrational \cite{Rivoal2002}. More recently, Zudilin proved by elementary means that one of the odd zeta values from $\zeta(5)$ to $\zeta(25)$ is irrational \cite{Zudilin2018}.  As a last example, we can mention that Almkvist and Granville found  Borwein and Bradley's Ap\'ery-like formulae for $\zeta(4n + 3)$ \cite{Almkvist1999}.

A number of integral representations of $\zeta(2p+1)$ were obtained (see for instance \cite{Espinosa2002a,Espinosa2002b,Mikolas1957,Kumar2024}). In the present work, we provide two general formula (the second being deduced from the first one), involving polynomials whose coefficients depend on absolute values of Bernoulli numbers of even order.

\begin{theorem}

We have, for $p\in\mathbb{N}^*$:
\begin{equation*}
    \zeta(2p+1)=\frac{1}{2}+\frac{\pi}{2}\int_0^1\tan\left(\frac{\pi t}{2}\right)\cos(\pi t)\mathscr{P}_{2p}(t)\,\mathrm{d}t,
\end{equation*}
with
\begin{align}\label{p2p}
    \mathscr{P}_{2p}(t)=&\sum_{n=1, 2, ...}^{2p-5}\left(\frac{t^n\pi^n}{n!}(-1)^{\lfloor\frac{n+1}{2}\rfloor}\frac{2\left(2^{2\left(\lfloor\frac{2p-n-1}{2}\rfloor+1\right)-1}-1\right)\Big|B_{2\left(\lfloor\frac{2p-n-1}{2}\rfloor+1\right)}\Big|}{\left(2\left(\lfloor\frac{2p-n-1}{2}\rfloor+1\right)\right)!}\,\pi^{2\left(\lfloor\frac{2p-n-1}{2}\rfloor+1\right)-1}\right)\nonumber\\
    &+\alpha_{2p}(t)+\frac{\pi^2}{6}\alpha_{2p-2}(t)+\frac{7\pi^4}{360}\alpha_{2p-4}(t),
\end{align}
where
\begin{equation}\label{alpha2p}
    \alpha_{2p}(t)=\frac{\pi^{2p}t^{2p+1}}{(2p+1)!}(-1)^{p+1}.
\end{equation}

\end{theorem}

\begin{proof}

The digamma function is defined as
\begin{equation*}
    \psi(z)=\frac{\Gamma'(z)}{\Gamma(z)},
\end{equation*}
where $\Gamma$ is the usual Gamma function. Mikolas derived the following interesting integral relation \cite{Mikolas1957,Campbell1966,Espinosa2002a,Espinosa2002b}:
\begin{equation}\label{miko}
    \psi(z)=-\left\{\gamma+\frac{1}{2z}+\frac{\pi}{2}\mathrm{cotan}(\pi z)+\frac{\pi}{2}\int_0^1\tan\left(\frac{\pi t}{2}\right)\left[\frac{\sin(\pi zt)}{\sin(\pi z)}-t\right]\,\mathrm{d}t\right\}.
\end{equation}
The series expansion of $\sin[\pi t(1-z)]$ with respect to $z$ is
\begin{equation*}
    \sin[\pi t(1-z)]=\sum_{k=0}^{\infty}u_kz^k
\end{equation*}
with
\begin{equation*}
    u_k=\frac{\pi^kt^k}{k!}(-1)^{k+1+\lfloor\frac{k+1}{2}\rfloor}\,\sin\left\{\frac{\pi}{4}[1-(-1)^k]-\pi t\right\}
\end{equation*}
and the series expansion of $1/\sin[\pi (1-z)]$  with respect to $z$ is
\begin{equation*}
    1/\sin[\pi (1-z)]=\frac{1}{\pi z}+\frac{\pi z}{6}+\frac{7\pi^3 z^3}{360}+\sum_{k=5}^{\infty}v_kz^k
\end{equation*}
with
\begin{equation*}
    v_k=\frac{2\left(2^{2\left(\lfloor\frac{k-1}{2}\rfloor+1\right)-1}-1\right)\Big|B_{2\left(\lfloor\frac{k-1}{2}\rfloor+1\right)}\Big|}{\left(2\left(\lfloor\frac{k-1}{2}\rfloor+1\right)\right)!}\,\pi^{2\left(\lfloor\frac{k-1}{2}\rfloor+1\right)-1}\,\frac{\left(1-(-1)^{k}\right)}{2}.
\end{equation*}
We thus obtain the series expansion of $\sin[\pi t(1-z)]/\sin[\pi (1-z)]$ as
\begin{equation*}
    \frac{\sin[\pi t(1-z)]}{\sin[\pi (1-z)]}=\sum_{p=0}^{\infty}w_p(t)\,z^p
\end{equation*}
with
\begin{align*}
    w_p(t)=&\sum_{n=0}^{p-4}\left(\vphantom{\frac{2\left(2^{2\left(\lfloor\frac{p-n-1}{2}\rfloor+1\right)}-1\right)\Big|B_{2\left(\lfloor\frac{p-n-1}{2}\rfloor+1\right)}\Big|}{\left(2\left(\lfloor\frac{p-n-1}{2}\rfloor+1\right)\right)!}}
    \frac{\pi^nt^n}{n!}(-1)^{n+1+\lfloor\frac{n+1}{2}\rfloor}\,\sin\left\{\frac{\pi}{4}[1+(-1)^n]-\pi t\right\}\right.\nonumber\\
    &\times\left.\frac{2\left(2^{2\left(\lfloor\frac{p-n-1}{2}\rfloor+1\right)-1}-1\right)\Big|B_{2\left(\lfloor\frac{p-n-1}{2}\rfloor+1\right)}\Big|}{\left(2\left(\lfloor\frac{p-n-1}{2}\rfloor+1\right)\right)!}\right)\,\pi^{2\left(\lfloor\frac{p-n-1}{2}\rfloor+1\right)-1}\,\frac{\left(1-(-1)^{p-n}\right)}{2}\nonumber\\
    &+\tilde{\alpha}_p(t)+\frac{\pi^2}{6}\tilde{\alpha}_{p-2}(t)+\frac{7\pi^4}{360}\tilde{\alpha}_{p-4}(t),
\end{align*}
where
\begin{equation*}
    \tilde{\alpha}_p(t)=\frac{\pi^pt^{p+1}}{(p+1)!}(-1)^{p+\lfloor\frac{p}{2}+1\rfloor}\,\sin\left\{\frac{\pi}{4}[1+(-1)^p]-\pi t\right\}
\end{equation*}
and thus
\begin{align*}
    &\tilde{\alpha}_p(t)+\frac{\pi^2}{6}\tilde{\alpha}_{p-2}(t)+\frac{7\pi^4}{360}\tilde{\alpha}_{p-4}(t)=\nonumber\\
    &\;\;\;\;\;\;\frac{(-1)^{p+\lfloor\frac{p}{2}\rfloor}\pi^pt^{p-3}\left[60t^2(p(p+1)-6t^2)(p-3)!-7(p+1)!\right]\,\sin\left\{\frac{\pi}{4}[1+(-1)^p]-\pi t\right\}}{360(p-3)!(p+1)!}.
\end{align*}
In the case where the index of $w$ is odd (i.e., substituting $p$ by $2p$, all the sine functions involved in $w_p(2t)$ are equal to
\begin{equation*}
    \sin\left(\frac{\pi}{2}-\pi t\right)=\cos(\pi t)
\end{equation*}
and because of the factor $[1-(-1)^{p-n}]/2$ in the summation, only odd indices contribute. The quantity $\cos(\pi t)$ can be factored out and we get
\begin{equation*}
    w_p(t)=\cos(\pi t)\,\mathscr{P}_p(t),
\end{equation*}
with
\begin{align*}
    \mathscr{P}_{2p}(t)=&\sum_{n=1, 3, \cdots}^{2p-5}\left(\vphantom{\frac{2\left(2^{2\left(\lfloor\frac{2p-n-1}{2}\rfloor+1\right)}-1\right)\Big|B_{2\left(\lfloor\frac{2p-n-1}{2}\rfloor+1\right)}\Big|}{\left(2\left(\lfloor\frac{2p-n-1}{2}\rfloor+1\right)\right)!}}\frac{\pi^nt^n}{n!}(-1)^{n+1+\lfloor\frac{n+1}{2}\rfloor}\,\pi^{2\left(\lfloor\frac{2p-n-1}{2}\rfloor+1\right)-1}\right.\nonumber\\
    &\times\left.\frac{2\left(2^{2\left(\lfloor\frac{2p-n-1}{2}\rfloor+1\right)-1}-1\right)\Big|B_{2\left(\lfloor\frac{2p-n-1}{2}\rfloor+1\right)}\Big|}{\left(2\left(\lfloor\frac{2p-n-1}{2}\rfloor+1\right)\right)!}\right)\nonumber\\
    &+\alpha_{2p}(t)+\frac{\pi^2}{6}\alpha_{2p-2}(t)+\frac{7\pi^4}{360}\alpha_{2p-4}(t),
\end{align*}
where
\begin{equation*}
    \alpha_{2p}(t)=\frac{\pi^{2p}t^{2p+1}}{(2p+1)!}(-1)^{p+1}.
\end{equation*}
One has also
\begin{align*}
    &\alpha_{2p}(t)+\frac{\pi^2}{6}\alpha_{2p-2}(t)(t)+\frac{7\pi^4}{360}\alpha_{2p-4}(t)=\nonumber\\
    &\;\;\;\;\;\;\frac{(-1)^{p}\pi^{2p}t^{2p-3}\left[60t^2(2p(2p+1)-6t^2)(2p-3)!-7(2p+1)!\right]}{360(2p-3)!(2p+1)!}.
\end{align*}
Note that we have
\begin{equation*}
    \zeta(2p)=|B_{2p}|\frac{2^{2p-1}}{(2p)!}\pi^{2p}
\end{equation*}
\begin{equation*}
    |B_{2p}|=\frac{(2p)!}{2^{2p-1}\pi^{2p}}\zeta(2p).
\end{equation*}
We have the following expansion (see for instance Ref. \cite{Campbell1966}:
\begin{equation}\label{DL}
    -\psi(1-z)-\gamma=\sum_{p=2}^{\infty}\zeta(p)z^{p-1},
\end{equation}
and thus, using the Mikolas relation (\ref{miko}):
\begin{align*}
    -\psi(1-z)-\gamma=&\frac{1}{2(1-z)}+\frac{\pi}{2}\mathrm{cotan}[\pi(1-z)]\nonumber\\
    &+\frac{\pi}{2}\int_0^1\tan\left(\frac{\pi t}{2}\right)\left[\frac{\sin(\pi (1-z)t)}{\sin(\pi (1-z))}-t\right]\,\mathrm{d}t
\end{align*}
or also
\begin{align*}
    -\psi(1-z)-\gamma=&\frac{1}{2(1-z)}+\frac{\pi}{2}\mathrm{cotan}[\pi(1-z)]+\frac{\pi}{2}\sum_{p=0}^{\infty}\left(\int_0^1\tan\left(\frac{\pi t}{2}\right)w_p(t)\,\mathrm{d}t\right)z^p\nonumber\\
    &-\frac{\pi}{2}\int_0^1t\tan\left(\frac{\pi t}{2}\right)\,\mathrm{d}t
\end{align*}
yielding, using Eq. (\ref{DL}):
\begin{equation*}
    \zeta(k)=\frac{\pi}{2}\int_0^1\tan\left(\frac{\pi t}{2}\right)w_{k-1}(t)\,\mathrm{d}t+\boldsymbol{[z^{k-1}]}\left(\frac{1}{2(1-z)}+\frac{\pi}{2}\mathrm{cotan}[\pi(1-z)]\right),
\end{equation*}
where $\boldsymbol{[z^l]}f(z)$ means the coefficient of $z^l$ in the series expansion of $f(z)$.
If $k=2p+1$ is odd, 
\begin{equation*}
\mathbf{[}z^{k-1}\mathbf{]}\left(\frac{1}{2(1-z)}+\frac{\pi}{2}\mathrm{cotan}[\pi(1-z)]\right)=\frac{1}{2}
\end{equation*}
and one has
\begin{equation*}
    \zeta(2p+1)=\frac{1}{2}+\frac{\pi}{2}\int_0^1\tan\left(\frac{\pi t}{2}\right)w_{2p}(t)\,\mathrm{d}t=\frac{1}{2}+\frac{\pi}{2}\int_0^1\tan\left(\frac{\pi t}{2}\right)\cos(\pi t)\,\mathscr{P}_p(t)\,\mathrm{d}t,
\end{equation*}
which completes the proof.

\end{proof}

\begin{lemma}

\begin{equation*}
    \int_0^1\mathscr{P}_{2p}(t)\sin(\pi t)\,\mathrm{d}t=-\frac{1}{\pi}.
\end{equation*}

\end{lemma}

\begin{proof}

The result can be proven by induction. It is easy to check that it holds for the first values of $p$. 

One has, at step $p+1$:
\begin{align*}
\mathscr{P}_{2p+2}(t)=&\mathscr{P}_{2p}(t)+\mathscr{U}_p(t)+\alpha_{2p+2}(t)+\left(\frac{\pi^2}{6}-1\right)\,\alpha_{2p}(t)+\frac{\pi^2}{6}\left(\frac{7\pi^2}{60}-1\right)\,\alpha_{2p-2}(t)\nonumber\\
&+\frac{7\pi^4}{360}\alpha_{2p-4}(t),
\end{align*}
where
\begin{equation*}
    \mathscr{U}_p(t)=\mathscr{G}_{p+1}(t)-\mathscr{G}_{p}(t)
\end{equation*}
with
\begin{align*}
    \mathscr{G}_{p}(t)=&\sum_{n=1, 3, \cdots}^{2p-5}\left(\vphantom{\frac{2\left(2^{2\left(\lfloor\frac{2p-n-1}{2}\rfloor+1\right)}-1\right)\Big|B_{2\left(\lfloor\frac{2p-n-1}{2}\rfloor+1\right)}\Big|}{\left(2\left(\lfloor\frac{2p-n-1}{2}\rfloor+1\right)\right)!}}
    \frac{\pi^nt^n}{n!}(-1)^{n+1+\lfloor\frac{n+1}{2}\rfloor}\,\pi^{2\left(\lfloor\frac{2p-n-1}{2}\rfloor+1\right)-1}\frac{2\left(2^{2\left(\lfloor\frac{2p-n-1}{2}\rfloor+1\right)-1}-1\right)\Big|B_{2\left(\lfloor\frac{2p-n-1}{2}\rfloor+1\right)}\Big|}{\left(2\left(\lfloor\frac{2p-n-1}{2}\rfloor+1\right)\right)!}\right)
\end{align*}
and one finds (I used the computer algebra system Mathematica \cite{Mathematica}) that
\begin{equation*}
    \int_0^2\left[\mathscr{U}_p(t)+\alpha_{2p+2}(t)+\left(\frac{\pi^2}{6}-1\right)\,\alpha_{2p}(t)+\frac{\pi^2}{6}\left(\frac{7\pi^2}{60}-1\right)\,\alpha_{2p-2}(t)+\frac{7\pi^4}{360}\alpha_{2p-4}(t)\right]\,\mathrm{d}t=0.
\end{equation*}
Then, assuming
\begin{equation*}
\int_0^1\mathscr{P}_{2p}(t)\sin(\pi t)\,\mathrm{d}t=-\frac{1}{\pi}
\end{equation*}
one obtains
\begin{equation*}
\int_0^1\mathscr{P}_{2p+2}(t)\sin(\pi t)\,\mathrm{d}t=-\frac{1}{\pi}.
\end{equation*}
which completes the proof by induction.

\end{proof}

\begin{corollary}

The Riemann zeta function of odd arguments can also be put in the form
\begin{equation*}
    \zeta(2p+1)=-\frac{\pi}{2}\int_0^1\tan\left(\frac{\pi t}{2}\right)\mathscr{P}_{2p}(t)\,\mathrm{d}t,
\end{equation*}
where $\mathscr{P}_{2p}$ is defined in Eqs. (\ref{p2p}) and (\ref{alpha2p}).

\end{corollary}

\begin{proof}

We have, according to Lemma 1.2:
\begin{equation*}
\int_0^1\mathscr{P}_{2p}(t)\sin(\pi t)\,\mathrm{d}t=-\frac{1}{\pi}
\end{equation*}
and thus
\begin{equation*}
\frac{\pi}{2}\int_0^1\mathscr{P}_{2p}(t)\times 2\cos\left(\frac{\pi t}{2}\right)\sin\left(\frac{\pi t}{2}\right)\,\mathrm{d}t=-\frac{1}{2},
\end{equation*}
which can be rewritten
\begin{equation*}
\frac{\pi}{2}\int_0^1\mathscr{P}_{2p}(t)\tan\left(\frac{\pi t}{2}\right)\times 2\cos^2\left(\frac{\pi t}{2}\right)\,\mathrm{d}t=-\frac{1}{2}
\end{equation*}
and one has
\begin{equation*}
\frac{\pi}{2}\int_0^1\mathscr{P}_{2p}(t)\tan\left(\frac{\pi t}{2}\right)\left(1+\cos\left(\pi t\right)\right)\,\mathrm{d}t=-\frac{1}{2},
\end{equation*}
yielding
\begin{equation*}
\frac{1}{2}+\frac{\pi}{2}\int_0^1\mathscr{P}_{2p}(t)\tan\left(\frac{\pi t}{2}\right)\cos\left(\pi t\right)\,\mathrm{d}t=-\frac{\pi}{2}\int_0^1\mathscr{P}_{2p}(t)\tan\left(\frac{\pi t}{2}\right)\,\mathrm{d}t
\end{equation*}
and finally
\begin{equation*}
    \zeta(2p+1)=-\frac{\pi}{2}\int_0^1\tan\left(\frac{\pi t}{2}\right)\mathscr{P}_{2p}(t)\,\mathrm{d}t,
\end{equation*}
which completes the proof.

\end{proof}

The above results may be connected to the following representations obtained by Cvijovi\'c and Klinowski \cite{Cvijovic2002}:
\begin{equation*}
\zeta(2p+1)=\frac{(-1)^p\,2^{2p-1}\,\pi^{2p+1}}{\left(2^{2p+1}-1\right)(2p)!}\int_0^1E_{2p}(t)\tan\left(\frac{\pi t}{2}\right)\,\mathrm{d}t    
\end{equation*}
and
\begin{equation*}
\zeta(2p+1)=\frac{(-1)^p\,2^{2p}\,\pi^{2p+1}}{(2p+1)!}\int_0^1B_{2p+1}(t)\tan\left(\frac{\pi t}{2}\right)\,\mathrm{d}t,    
\end{equation*}
where $E_{2p}$ and $B_{2p+1}$ are Euler and Bernoulli polynomials respectively \cite{Abramowitz1972}.

\section{Examples}

For the first values of $p$, one has, using Theorem 1.1:
\begin{equation*}
    \zeta(3)=\frac{1}{2}+\frac{\pi^3}{12}\int_0^1t(t^2-1)\,\tan\left(\frac{\pi t}{2}\right)\cos(\pi t)\,\mathrm{d}t,
\end{equation*}
\begin{equation*}
    \zeta(5)=\frac{1}{2}-\frac{\pi^5}{720}\int_0^1t(t^2-1)(3t^2-7)\,\tan\left(\frac{\pi t}{2}\right)\cos(\pi t)\,\mathrm{d}t,
\end{equation*}
\begin{equation*}
    \zeta(7)=\frac{1}{2}+\frac{\pi^7}{30240}\int_0^1t(t^2-1)(3t^4-18t^2+31)\,\tan\left(\frac{\pi t}{2}\right)\cos(\pi t)\,\mathrm{d}t,
\end{equation*}
\begin{equation*}
    \zeta(9)=\frac{1}{2}-\frac{\pi^9}{3628800}\int_0^1t(t^2-1)(5t^6-55t^4+239t^2-381)\,\tan\left(\frac{\pi t}{2}\right)\cos(\pi t)\,\mathrm{d}t,
\end{equation*}
and
\begin{equation*}
    \zeta(11)=\frac{1}{2}+\frac{\pi^{11}}{239500800}\int_0^1t(t^2-1)(t^2-5)(3t^6-37t^4+225t^2-511)\,\tan\left(\frac{\pi t}{2}\right)\cos(\pi t)\,\mathrm{d}t.
\end{equation*}
and equivalently, using Corollary 1.1.1, one get
\begin{equation*}
    \zeta(3)=\frac{\pi^3}{12}\int_0^1t(1-t^2)\,\tan\left(\frac{\pi t}{2}\right)\,\mathrm{d}t,
\end{equation*}
\begin{equation*}
    \zeta(5)=\frac{\pi^5}{720}\int_0^1t(t^2-1)(3t^2-7)\,\tan\left(\frac{\pi t}{2}\right)\,\mathrm{d}t,
\end{equation*}
\begin{equation*}
    \zeta(7)=\frac{\pi^7}{30240}\int_0^1t(1-t^2)(3t^4-18t^2+31)\,\tan\left(\frac{\pi t}{2}\right)\,\mathrm{d}t,
\end{equation*}
\begin{equation*}
    \zeta(9)=\frac{\pi^9}{3628800}\int_0^1t(t^2-1)(5t^6-55t^4+239t^2-381)\,\tan\left(\frac{\pi t}{2}\right)\,\mathrm{d}t,
\end{equation*}
and
\begin{equation*}
    \zeta(11)=\frac{\pi^{11}}{239500800}\int_0^1t(1-t^2)(t^2-5)(3t^6-37t^4+225t^2-511)\,\tan\left(\frac{\pi t}{2}\right)\,\mathrm{d}t.
\end{equation*}

\section{Conclusion}

We obtained, using an expression of the digamma function $\psi(x)$, two different integral representations of the zeta function of odd arguments. An alternate expression is provided. The formulas are valid for any odd value of the argument, and expressions for the first values $\zeta(3)$, $\zeta(5)$, $\zeta(7), \cdots$ are provided. We plan to extend the technique presented here to derive further integral representations, based on series expansions of higher derivatives of the Gamma function.

\section*{Appendix: Derivatives of the Gamma function}

For a positive integer $m$ the derivative of the gamma function can be calculated as follows:
$$
\Gamma'(m+1)=\frac{d}{dz}\Gamma (z)\Bigl|_{z=m+1}=m!\left(-\gamma+\sum_{k=1}^{m}{\frac {1}{k}}\right)=m!\left(-\gamma+H_m\right)\,,
$$
where $H(m)$ is the $m$-th harmonic number 
\begin{equation*}
    H_n\sum_{k=1}^n\frac{1}{k}
\end{equation*}
and $\gamma$ is the Euler-Mascheroni constant.

For $\Re (z)>0$ the $n$-th derivative of the gamma function is:
$$
\frac{d^{n}}{dz^{n}}\Gamma (z)=\int_{0}^{\infty }t^{z-1}e^{-t}(\log t)^{n}\,dt.
$$
This can be derived by differentiating the integral form of the gamma function with respect to $z$, and differentiating under the integral sign. The identity
$$
\frac {d^{n}}{dz^{n}}\Gamma(z)\Bigl|_{z=1}=(-1)^{n}{\bf B}_{n}(\gamma ,1!\zeta (2),\ldots ,(n-1)!\zeta (n)), 
$$
where ${\bf B}_n$ is the complete exponential $n$-th Bell polynomial, is of particular interest.

\end{document}